\documentclass[10pt]{amsart}

\usepackage[T1]{fontenc}
\usepackage[utf8]{inputenc}
\usepackage{graphicx}
\usepackage{amssymb}
\usepackage{amsthm}
\usepackage{hyperref}
\usepackage{enumerate}

\theoremstyle{plain} 
\newtheorem{theorem}{Theorem}
\newtheorem{lemma}{Lemma}
\newtheorem{corollary}{Corollary}
\newtheorem{ex}{Example}

\begin{document}

\title[B\"{o}hm 1968]{Some properties of $\beta$-$\eta$-normal forms
  in $\lambda$-K-calculus \\
  (Alcune propriet{\`a} delle forme $\beta$-$\eta$-normali nel $\lambda$-K-calcolo)}

\author{Corrado B{\"o}hm}

\begin{abstract}
  This is a plain English translation of~\cite{Boehm:1968},
  originally in Italian, by Chun Tian. All footnotes (and
  citations only found in footnotes, of course) are added by the translator.
\end{abstract}

\newenvironment{nouppercase}{%
  \let\uppercase\relax%
  \renewcommand{\uppercasenonmath}[1]{}}{}
\begin{nouppercase}
\maketitle
\end{nouppercase}

\section{Introduction}

We assume the reader is familiar with the theory of $\lambda$-k-calculus
described in~\cite{Church:1951}, \cite{Rosenbloom:1950}, \cite{Curry:1958}, or
at least some work of the author~\cite{Boehm:1966a},~\cite{Boehm:1966b}.

In the following we will adopt ``combinator'' as a synonym of formulas
of the $\lambda$-k-calculus without free variables.

And note that combinators form a set $\mathcal{C}$ of entities,
which can also be interpreted as functions from $\mathcal{C}$ to
$\mathcal{C}$, or from $\mathcal{C}\times\mathcal{C}$ to
$\mathcal{C}$, or from $\mathcal{C}\times\mathcal{C}\times\mathcal{C}$
to $\mathcal{C}$, etc.

A subclass of combinators is constructed from the so-called
\emph{normal forms}, whose algebraic characteristic is arduous
(difficult to obtain).

And note that the set $\mathcal{C}$ of combinators is generated from
two normal forms distinguished among them
\begin{align*}
  \mathrm{S} &\equiv \lambda x\lambda y\lambda z\,( x z\,( y z )\,) \\
  \mathrm{K} &\equiv \lambda x\lambda y\,x
\end{align*}
with repeatedly ``applications''.
Now a significance which can be attributed to the Corollary \ref{cor:1} of this
work is the following:
there exist no privileged normal forms (which could be considered as
$\mathrm{S}$ and $\mathrm{K}$) that, if one establish the congruence
classes of combinators (with respect to applications) such that
at least two non-congruent combinators exist in them, this is
enough to state that each class cannot contain more than one normal
form.

A second significance renders the Corollary \ref{cor:1} rather peculiar
(unusual). In fact, considering normal forms as functions of a
finite number of variables, one paraphrase (explanation) of the result
is: given two distinct functions it is always possible to find a
$n$-tuple of arguments such that the values of the two functions for
that $n$-tuple of arguments coincide with two pre-given arbitrary
values.

The Corollary \ref{cor:1} resolves, among others, a problem posted
in~\cite{Boehm:1966b} (pp.~194): Finding an algorithm for constructing, given
two distinct normal forms $F$ and $G$, a combinator $\Delta$ such that
\begin{equation*}
  \Delta F = \mathrm{K}, \qquad \Delta G = \mathrm{0}.
  \qquad(\mathrm{0} \equiv \lambda x\lambda y\, y)
\end{equation*}

This algorithm is implicit in the proof of Theorem \ref{thm:1} of the present work.

It is easy to show that the combinators form a semigroup with left zeros,
and an identity with respect to all composition ($\circ$) operations with the
generators $C_*$, $C_*\mathrm{S}$, $C_*\mathrm{K}$, where
\begin{equation*}
  C_* \equiv \lambda x\lambda y\,(y\,x)
\end{equation*}

The Corollary 2 affirms that, for any normal form $F$, which is not
zero of the semigroup (thus is different from $\lambda y\,X$, where
$y$ does not occur in $X$), the following property
\begin{equation*}
  \Delta\circ F\circ\nabla = \mathrm{I}\qquad\qquad
  (\mathrm{I}\equiv\lambda x\,x,\;\text{identity combinator})
\end{equation*}
holds for appropriate combinators $\Delta$ and $\nabla$ of which the
structure will be given.

Lastly the Corollary 3 concerns the cardinality of the codomain of the
combinators considered as unary functions. For zeros of the
semigroup, i.e.~the combinators $\lambda y\,X$, their cardinality is one.
The Corollary 3 proves that if two distinct normal forms fall in the
codomain, the cardinality must be at least three.

This corollary is susceptible (easily allowed) for improvements. In a next work we
intende to replace with the following: if $n$ distinct normal forms
fall in the codomain of a combinator, the cardinality of the codomain
is at least $n+1$.

We propose the following conjecture: the cardinality of the codomain
of any combinator is either 1 or $\infty$.

\section{Syntactic Notations}

As we know the set $\mathcal{N}$ of the $\beta$-$\eta$-normal forms is
recursively defined as follows:
\begin{align*}
  X_1,\ldots X_m \in \mathcal{N} &\rightarrow \xi\,X_1\cdots X_m \in
                                   \mathcal{N},\qquad m\geqslant
                                   0, && \text{$\xi$ variable} \\
  X \in \mathcal{N} &\rightarrow \lambda t\,X \in
                      \mathcal{N}, && \text{$t$ variable}
\end{align*}
(except the case $X\equiv X'\, t$ where $t$ does not
occur\footnote{This is to prevent $\lambda t\,X'\,t \in\mathcal{N}$,
  which can be $\eta$-reduced to $X'$.} in $X'$.

From here one can deduce that every formula $N \in\mathcal{N}$ has the
following structure
\begin{equation*}
  N \equiv \lambda t_1\,\ldots \lambda t_n\,(\xi\,X_1\cdots X_m)
\end{equation*}
where $n\geqslant 0$ is the \emph{order} of $N$, $m\geqslant 0$ is the
\emph{grade} of $N$, $X_k\in\mathcal{N}$, and $\xi$ is \emph{principle
variable} of $N$.

Note that the symbol $\equiv$ represents the identity of formulas or
definitional identity, while the symbol $=$ is going for the
$\lambda$-k-convertibility of formulas.
As is usual if $M$ is a formula and if $y_1,y_2,\ldots$ are variables,
the writing $M[y_1,y_2,\ldots]$ indicates that in the formula $M$ the
variables $y_1,y_2$ can occur freely. The writing $M[G_1,G_2,\ldots]$
indicates the result of the substitution in $M$ of every free
occurrence of $y_1$ with $G_1$, $y_2$ with $G_2$, etc.

Now we can enunciate (state) the following theorem:

\begin{theorem}
  \label{thm:1}
  Let $N_1,N_2\in\mathcal{N}$. If $N_1\neq N_2$ then there exist
  \begin{enumerate}
  \item Two integers $t,s \geqslant 0$
  \item Combinators $\nabla_k\in\mathcal{N}, k = 1,\ldots t$
  \item Formulas $H_k\in\mathcal{N}, k = 1,\ldots s$
  \end{enumerate}
  such that
  \begin{equation*}
    N_i[\nabla_1\,y_1,\ldots \nabla_t\,y_t]\,H_1\ldots H_s = v_i\qquad
    (i = 1,2)
  \end{equation*}
  where the variables $v_i$ occur free in some $H_k$.
\end{theorem}

A very significant special case is the following:
\begin{corollary}
  \label{cor:1}
  Let $C_1,C_2\in\mathcal{N}$ be two combinators in normal form. If
  $C_1\neq C_2$ then there exist
  \begin{enumerate}
  \item An integer $s > 0$,
  \item Combinators $G_1,\ldots G_s$ having the structure
    $\mathrm{K}^a\mathrm{K}^b$ or $\mathrm{K}^aX_1$,
    $\mathrm{K}^bX_2$, where $a,b \geqslant 0$ are integers.
  \end{enumerate}
  such that
  \begin{equation*}
    C_i\,G_1\ldots G_s = X_i\qquad(i = 1,2)
  \end{equation*}
  where $X_1,X_2$ are two arbitrary combinators.
\end{corollary}

The corollary is obtained from the theorem by specifying the structure
of the formulas of which the proof makes use, and by observing that,
if in $N_1,N_2$ no variable occurs freely, we have
$N_i[y_1,\ldots y_t] \equiv N_i[\nabla_1 y_1,\ldots,\nabla_t y_t]
\equiv C_i$ and that, if $H_k\equiv H_k[v_1]$ or $H_k\equiv H_k[v_2]$,
it is sufficient to define
\begin{equation*}
  G_k \equiv H_k[X_i]
\end{equation*}
and that, if $H_k\equiv H_k[x]$, it is sufficient to define
\begin{equation*}
  G_k \equiv H_k[\mathrm{I}]
\end{equation*}

\section{Proof method of the Theorem \ref{thm:1}}

An equivalence relation $\sim$ is introduced such that if two formulas
are inequivalent the proof is almost immediate (Lemma 1).
If two distinct formulas are equivalent, a finite chain of pairs of equivalent
formulas is associated with that pair, except for the last pair which
is formed by inequivalent formulas.
It is shown that if the principal variable of each pair (or the two of
the last one) is different from that of every other pair, then it is
possible to pass from the initial pair to the inequivalent pair with
the operations allowed by the theorem (Lemma 2).
Finally, it is shown (Lemma 3) that there exists a substitution of
variables which essentially transforms the finite chain, relating to
the given pair of distinct formulas, into a chain of the same length
in which all the principal variables of the pairs are different from
each other.

\section{An equivalence relation}

Let $N_1,N_2\in\mathcal{N}$ and $N_1\ne N_2$. We have seen that there
must be
\begin{equation*}
  N_i \equiv \lambda t_1\ldots \lambda t_{n_i}\left ( \xi_i\,
    X_1^{(i)} \cdots X_{m_i}^{(i)} \right )\qquad (i = 1,2)
\end{equation*}
where $n_i \geqslant 0, m_i \geqslant 0, X_j^i\in\mathcal{N}$, and
$\xi_i$ are variables.

Let $\overline{n} = \max\,n_i$. Given a number $f \geqslant
\overline{n}$ and the pair $(N_1,N_2)$, it is uniquely determined the
pair $(M_1,M_2)$ defined by
\begin{equation}
  \label{eq:1}
  M_i = N_i\,x_1x_2\ldots x_f \qquad (i = 1,2)
\end{equation}
where $x_1,x_2,\ldots, x_f$ do not occur in $N_i$.

Naturally $M_i\in\mathcal{N}$ and from a personal note of
$\lambda$-calculus one has
\begin{equation*}
  N_1 \ne N_2 \,\leftrightarrow\, M_1 \ne M_2
\end{equation*}

Furthermore there must be
\begin{equation}
  \label{eq:1'}
  M_i \equiv \xi'_i\,\underline{X}_1^{(i)} \cdots
  \underline{X}_{m_i}^{(i)}\,x_{n_i+1}\ldots x_f\qquad (i = 1,2)
\end{equation}
where $\underline{X}_{m_i}^{(i)}, \xi'_i$ are the result of
substitutions\footnote{Hereafter the notation of substitution has been
  changed from $[t_k / x_k]$ to $[t_k \mapsto x_k]$ for clarity.}
$[t_k \mapsto x_k]$ in $X_j^{(i)}, \xi_i$.

We define the following relation among normal formulas:
\begin{equation*}
  N_1 \sim N_2 \,\leftrightarrow\,
  \xi_1 = \xi_2\;\wedge\; n_1 - m_1 = n_2 - m_2
\end{equation*}

Obviously this is an equivalence relation, therefore
\begin{align*}
  N_1 \not\sim N_2 &\rightarrow N_1 \ne N_2\\
  N_1 \sim N_2 &\rightarrow M_1\sim M_2
\end{align*}

Let us rewrite (\ref{eq:1}) in a more uniform way:
\begin{equation}
  \label{eq:2}
  M_i \equiv z_i\,Y_1^{(i)}\cdots Y_j^{(i)}\cdots
  Y_{g_i}^{(i)}\qquad(i = 1,2)
\end{equation}
where naturally $z_i \equiv \xi'_i, \underline{X}_k^{(i)} \equiv
Y_k^{(i)}, x_{n_i+h} \equiv Y_{m_i+h}$ and that
\begin{equation*}
  g_i = m_i + f - n_i\qquad (f\geqslant\overline{n})
\end{equation*}

It is obvious that
\begin{equation*}
  M_1 \sim M_2 \,\leftrightarrow\, z_1 = z_2 \,\wedge\, g_1 = g_2.
\end{equation*}

From the above the following statement is obvious:

If $N_1 \ne N_2$ and $N_1 \sim N_2$, then, putting $g = m_i +
\overline{n} - n_i$, that is $f = \overline{n}$, in (\ref{eq:2}), we
get $Y_j^{(1)} \ne Y_j^{(2)}$ for at least one value of the index $j$,
$1 \leqslant j \leqslant g$.

\begin{lemma}
  \label{lem:1}
  Theorem \ref{thm:1} holds when $N_1 \not\sim N_2$.
\end{lemma}

\begin{proof}
  It is sufficient to check the alternative cases:
  \begin{enumerate}[(i)]
  \item $\xi_1 \ne \xi_2$
  \item $\xi_1 = \xi_2 \;\wedge\; n_1 - m_1 \ne n_2 - m_2$
  \end{enumerate}

In the case (i) we consider the formulas $M_i\,v_1 v_2\,[z_1,z_2]$ where
$M_i$ are defined in (\ref{eq:1}), (\ref{eq:1'}) and (\ref{eq:2}), and
$f = \overline{n}$ is chosen. It is easy to verify that
\begin{equation}
  \label{eq:3}
  M_i\,v_1
  v_2\,[\mathrm{K}^{1+g_1}\,\mathrm{K}\,z_1,\mathrm{K}^{2+g_2}\,\mathrm{I}\,z_2]
  = v_i\qquad(i = 1,2)
\end{equation}

In the case (ii) $f = 1 + \overline{n} + p$ is chosen, where $p =
|(n_1 - m_1) - (n_2 - m_2)|$ and let $\overline{g} =
\underset{i}{\max}\,(m_i + \overline{n} - n_i)$.

As a consequence of the choice of $f$ we obtain in (\ref{eq:2}) that
\begin{equation*}
  Y_{1+\overline{g}}^{(1)} = x_{1+\overline{n}}\qquad,\qquad
  Y_{1+\overline{g}}^{(2)} = x_{1+\overline{n} +p}
\end{equation*}
or
\begin{equation*}
  Y_{1+\overline{g}}^{(1)} = x_{1+\overline{n} +p}\qquad,\qquad
  Y_{1+\overline{g}}^{(2)} = x_{1+\overline{n}}\enspace.
\end{equation*}

We consider the formulas
$M_i\,[x_{1+\overline{n}},x_{1+\overline{n}+p},z]$ where $z = z_1 =
z_2$. It is easy to verify that
\begin{equation}
  \label{eq:4}
  M_i\,[\mathrm{K}^p v_1,v_2,\mathrm{K}^{1+\overline{g}}\,\mathrm{I}\,z]
  = v_i\quad\text{or}\quad v_{3-i}\qquad(i = 1,2)
\end{equation}

The formulas (\ref{eq:3}) and (\ref{eq:4}) essentially prove
Lemma~\ref{lem:1}.
In fact, they summarise four possible cases for (\ref{eq:3}):
$\xi_1,\xi_2$ both free; $\xi_1$ free and $\xi_2$ bound; $\xi_1$ bound
and $\xi_2$ free; $\xi_1$ and $\xi_2$ both bound. And for (\ref{eq:4})
two possible cases: $\xi$ free; $\xi$ bound.

\paragraph{Case 1}
Let $\xi_1 \equiv y_1$, $\xi_2 \equiv y_2$. Then
(\ref{eq:3}) can be rewritten to
\begin{equation*}
  N_i\,[\nabla_1y_1,\nabla_2y_2]\,x_1\ldots
  x_{\overline{n}}\,v_1v_2 = v_i\qquad (i = 1,2)
\end{equation*}
where $\nabla_1 \equiv \mathrm{K}^{g_1+1}\mathrm{K}$,
$\nabla_2 \equiv \mathrm{K}^{g_2+2}\mathrm{I}$.

\paragraph{Case 2 (and similarly Case 3)}
Let $\xi_1 \equiv y_1$, $\xi_2 \equiv t_k$. Then (\ref{eq:3}) can be
rewritten to
\begin{equation*}
  N_i\,[\nabla_1y_1]\,x_1\ldots x_{k-1}(\nabla_2 x_k)\,x_{k+1}\ldots
  x_{\overline{n}}\,v_1v_2 = v_i\qquad (i = 1,2) 
\end{equation*}

\paragraph{Case 4}
Let $\xi_1 \equiv t_k$, $\xi_2 \equiv t_l$, $l \ne k$. Then we have
\begin{equation*}
  N_i\,x_1\ldots x_{k-1}(\nabla_1 x_k)\,x_{k+1}\ldots
  x_{l-1}\, (\nabla_2 x_l)\,x_{l+1}\ldots x_{\overline{n}}\,v_1v_2 = v_i\qquad (i = 1,2)
\end{equation*}

\paragraph{Case 5}
Let $\xi = y_1$ free. Then (\ref{eq:4}) can be rewritten to
\begin{equation*}
  N_i\,[\nabla_1y_1]\,x_1\ldots
  x_{\overline{n}}\,(\mathrm{K}^p v_1)\,
  x_{\overline{n}+2}\ldots x_{\overline{n}+p}\,v_2 =
  v_i\quad\text{or}\quad v_{3-i}\qquad(i = 1,2)
\end{equation*}
where $\nabla_1 \equiv \mathrm{K}^{\overline{g}+1}\mathrm{I}$.

\paragraph{Case 6}
Let $\xi = t_k$ bound. Then we have
\begin{equation*}
  N_i\,x_1\ldots x_{k-1}\,(\nabla_1 x_k)\,x_{k+1}\ldots
  x_{\overline{n}}\,(\mathrm{K}^p v_1)\,x_{\overline{n}+2}\,v_2 =
  v_i\quad\text{or}\quad v_{3-i}\qquad(i = 1,2)
\end{equation*}

The previous relations represent the thesis of Theorem~\ref{thm:1} (up to a
possible exchange between $v_1$ and $v_2$ in the case (ii)), therefore they prove
Lemma~\ref{lem:1}.
\end{proof}

Now we want to examine the case $N_1 \ne N_2$, $N_1 \sim N_2$.

From the previous statement, Lemma~\ref{lem:1} follows that, for $f =
\overline{n}$, there exists a pair $(N_1^{(1)}, N_2^{(1)})$ of
subformulas $N_1^{(1)} \ne N_2^{(1)}$ of the pair $(M_1,M_2)$
uniquely associated to $(N_1,N_2)$.
For $(N_1^{(1)}, N_2^{(1)})$, either $N_1^{(1)} \not\sim N_2^{(1)}$ or
$N_1^{(1)} \sim N_2^{(1)}$ is true.
In the latter case, repeating the reasoning done for $(N_1,N_2)$ we
can uniquely define $(M_1^{(1)},M_2^{(1)})$, and consider a pair of
subformulas $N_1^{(2)} \ne N_2^{(2)}$ and so on.
The following statement therefore applies:

To the pair $(N_1^{(0)}, N_2^{(0)})$ with $N_1^{(0)} \ne N_2^{(0)}$,
it is naturally associated with an integer $0 \leqslant \sigma
\leqslant \infty$ and a chain of formulas $(M_1^{(l)}, M_2^{(l)})$
($l = 0,\ldots,\sigma$) all having the structure (\ref{eq:2}) and such
that the following relations hold
\begin{equation*}
  M_1^{(l)} \sim M_2^{(l)}\quad l = 0,\ldots,\sigma-1,\quad
  M_1^{(\sigma)} \not\sim M_2^{(\sigma)}\enspace.
\end{equation*}

The finiteness of $\sigma$ is a consequence of the fact that, in the
transition from $(N_1^{(l)}, N_2^{(l)})$  to $(N_1^{(l+1)},
N_2^{(l+1)})$, in at least one of the two formulas constituting the
pair, the number of symbols decreases.

\begin{ex}
  \label{ex:1}
  Let
  \begin{align*}
    N_1^{(0)} &\equiv \lambda t_1\lambda t_2\lambda
                t_3\,(t_1\lambda u\,t_2\lambda v\,(t_2\,t_1\lambda
                z\lambda s\,(z\,v))\,t_3) \\
    N_2^{(0)} &\equiv \lambda u\,\lambda v\,((u\,\lambda t\,v)\,
                (\lambda x\,(v\,u\,\lambda z\,z)))
  \end{align*}
  We have
  \begin{align*}
    M_1^{(0)} &= N_1^{(0)}x_1x_2x_3 = x_1\lambda u\,x_2\lambda v\,(x_2\,x_1\lambda
                z\lambda s\,(z\,v))\,x_3 \\
    M_2^{(0)} &= N_2^{(0)}x_1x_2x_3 = ((x_1\lambda t\,x_2)\,
                (\lambda x\,(x_2x_1\lambda z z)))\,x_3 \\
    N_1^{(1)} &\equiv \lambda v\,(x_2x_1\lambda z\lambda s\,(z\,v)) \\
    N_2^{(1)} &\equiv \lambda x\,(x_2x_1\lambda z\,z)
  \end{align*}
  \begin{align*}
    M_1^{(1)} &= N_1^{(1)}x_4 = x_2x_1\lambda z\lambda s\,(z\,x_4) \\
    M_2^{(1)} &= N_2^{(1)}x_4 = x_2x_1\lambda z\,z \\
    N_1^{(2)} &\equiv \lambda z\,\lambda s\,(z\,x_4) \\
    N_2^{(2)} &\equiv \lambda z\,z \\
    M_1^{(2)} &= N_1^{(2)}x_5x_6 = x_5x_4 \\
    M_2^{(2)} &= N_2^{(2)}x_5x_6 = x_5x_6 \\
    N_1^{(3)} &\equiv M_1^{(3)} \equiv x_4 \\
    N_2^{(3)} &\equiv M_2^{(3)} \equiv x_6
  \end{align*}
\end{ex}

Now we proceed to the proof of the following

\begin{lemma}
  \label{lem:2}
  If $N_1^{(0)} \ne N_2^{(0)}$ and all the principle variables
  $z^{(0)},\ldots,z^{(\sigma-1)},z_i^{(\sigma)}$ of a chain
  $(M_1^{(l)},M_2^{(l)})$ ($\,l = 0,\ldots,\sigma$) associated with
  $(N_1^{(0)},N_2^{(0)})$ are distinct from each other (but it could
  be that $z_1^{(\sigma)} = z_2^{(\sigma)}$), then Theorem~\ref{thm:1} is
  valid for $N_1^{(0)}, N_2^{(0)}$.
\end{lemma}

\begin{ex}[Continuation of Example~\ref{ex:1}]
  In this case $\sigma = 3$ and also $z^{(0)} \equiv x_1$, $z^{(1)}
  \equiv x_2$, $z^{(2)} \equiv x_5$, $z_1^{(3)} \equiv x_4$,
  $z_2^{(3)} \equiv x_6$.
  The premises of Lemma~\ref{lem:2} are verified. Furthermore
  \begin{align*}
    N_i^{(1)}\quad &\text{is}\quad Y_2^{(i)}\quad\text{of}\quad
                     M_i^{(0)}\,[z^{(0)},\ldots,z_1^{(3)},
                     z_2^{(3)}],\qquad g = 3 \\
    N_i^{(2)}\quad &\text{is}\quad Y_2^{(i)}\quad\text{of}\quad
                     M_i^{(1)}\,[z^{(0)},\ldots,z_1^{(3)},
                     z_2^{(3)}],\qquad g = 2 \\
    N_i^{(3)}\quad &\text{is}\quad Y_1^{(i)}\quad\text{of}\quad
                     M_i^{(2)}\,[z^{(0)},\ldots,z_1^{(3)},
                     z_2^{(3)}],\qquad g = 1
  \end{align*}

  It is easy to verify, taking also account of Lemma~\ref{lem:1}
  for $z_1^{(3)}, z_2^{(3)}$, that
  \begin{align*}
    M_i^{(0)}\,[\mathrm{K}^2\mathrm{K}\,z^{(0)},
    \mathrm{K}^2\mathrm{I}\,z^{(1)},
    \mathrm{K}\mathrm{I}\,z^{(2)},
    \mathrm{K}\mathrm{K}\,z_1^{(3)},
    \mathrm{K}^2\mathrm{I}\,z_2^{(3)} ] &=
    N_i^{(1)}\,[\mathrm{K}^2\mathrm{K}\,z^{(0)},\ldots,\mathrm{K}^2\mathrm{I}\,z_2^{(3)}] \\
    M_i^{(1)}\,[\mathrm{K}^2\mathrm{K}\,z^{(0)},
    \mathrm{K}^2\mathrm{I}\,z^{(1)},
    \mathrm{K}\mathrm{I}\,z^{(2)},
    \mathrm{K}\mathrm{K}\,z_1^{(3)},
    \mathrm{K}^2\mathrm{I}\,z_2^{(3)} ] &=
    N_i^{(2)}\,[\mathrm{K}^2\mathrm{K}\,z^{(0)},\ldots,\mathrm{K}^2\mathrm{I}\,z_2^{(3)}] \\
    M_i^{(2)}\,[\mathrm{K}^2\mathrm{K}\,z^{(0)},
    \mathrm{K}^2\mathrm{I}\,z^{(1)},
    \mathrm{K}\mathrm{I}\,z^{(2)},
    \mathrm{K}\mathrm{K}\,z_1^{(3)},
    \mathrm{K}^2\mathrm{I}\,z_2^{(3)} ] &=
    N_i^{(3)}\,[\mathrm{K}^2\mathrm{K}\,z^{(0)},\ldots,\mathrm{K}^2\mathrm{I}\,z_2^{(3)}]
  \end{align*}
  from which we finally get
  \begin{equation*}
    N_i^{(0)}\,(\mathrm{K}^2\mathrm{K}\,x_1)\,(\mathrm{K}^2\mathrm{I}\,x_2)\,x_3\,
    (\mathrm{K}\mathrm{K}\,x_4)\,(\mathrm{K}\mathrm{I}\,x_5)\,
    (\mathrm{K}^2\mathrm{I}\,x_6)\,v_1v_2 = v_i,
  \end{equation*}
  or even
  \begin{equation*}
    N_i^{(0)}\,(\mathrm{K}\mathrm{K})\,(\mathrm{K}\mathrm{I})\,\mathrm{I}\mathrm{K}\mathrm{I}\,
    (\mathrm{K}\mathrm{I})\,v_1v_2 = v_i\enspace.
  \end{equation*}
\end{ex}

\begin{proof}
  For the hypothesis $l \ne k \rightarrow z^{(l)} \ne z^{(k)},
  z_i^{(\sigma)} \ne z^{(k)}$ ($l,k = 0,\ldots,\sigma, i = 1,2$),
  where $z^{(l)}$ are free in $M_i^{(l)} \equiv
  M_i^{(l)}\,[z^{(0)},\ldots z^{(\sigma-1)}, z_1^{\sigma},
  z_2^{\sigma}]$, $l = 0,\ldots,\sigma$.
  Now let $j_l$ be the value of the index $j$ in $M_i^{(l)}$
  identifying $N_i^{(l+1)}$, $l = 0,\ldots,\sigma-1$.

  We have the following relations:
  \begin{equation}
    \label{eq:5}
    N_i^{(l)}\,x_{r_l+1}\ldots x_{r_l+\overline{n}_l} =
    M_i^{(l)}\,[z^{(0)},\ldots,z^{(\sigma-1)},z_1^{(\sigma)},z_2^{(\sigma)}]
  \end{equation}
  where
  \begin{equation*}
    r_l = \sum_0^{l-1}\,{\overline{n}_k}\quad,\quad \overline{n}_l = \underset{i}{\max}\,n_i^{(l)}
  \end{equation*}

  \begin{equation}
    \label{eq:6}
    \begin{split}
    & M_i^{(l)}\,[\nabla^{(0)}z^{(0)},\ldots,\nabla^{(\sigma-1)}z^{(\sigma-1)},\nabla_1^{(\sigma)}z_1^{(\sigma)},
    \nabla_2^{(\sigma)}z_2^{(\sigma)}] = \\
    & = N_i^{(l+1)}\,[\nabla^{(0)}z^{(0)},\ldots,\nabla^{(\sigma-1)}z^{(\sigma-1)},\nabla_1^{(\sigma)}z_1^{(\sigma)},
    \nabla_2^{(\sigma)}z_2^{(\sigma)}]
    \end{split}
  \end{equation}
where $\nabla^{(l)} = \mathrm{K}^{j_l}\mathrm{K}^{g_l-j_l}$, $g_l =
m_i^{(l)} + \overline{n}_l - n_i^{(l)}$, $l = 0,\ldots,\sigma-1$, and
where $\nabla_1^{(\sigma)}, \nabla_2^{(\sigma)}$ are obtained from
formulas (\ref{eq:3}) or (\ref{eq:4}) of Lemma~\ref{lem:1}.

The lemma is true for $\sigma = 0$, because in this case $N_1^{(0)}
\not\sim N_2^{(0)}$ and the thesis is those of Lemma~\ref{lem:1}.
Now let $\sigma > 0$. We have
\begin{equation}
  \label{eq:7}
  M_1^{(\sigma)}\,[\ldots,\nabla^{(l)} z^{(l)},\ldots] \not\sim
  M_2^{(\sigma)}\,[\ldots,\nabla^{(l)} z^{(l)},\ldots]
\end{equation}
in which all substitutions $[z^{(l)} \mapsto \nabla^{(l)} z^{(l)}]$ with $l \ne
\sigma$,
being $z^{(l)} \ne z_i^{(\sigma)}$, can alter neither the number nor
the position of the components of $M_i^{(\sigma)}$, which remain
normal.
Even if subformulas of $M_i^{(\sigma)}$ in corresponding positions
become equal, the relation $\not\sim$ is preserved since it is defined
independently of this.

It is deduced from (\ref{eq:7}) that
$N_1^{(\sigma)}\,[\ldots,\nabla^{(l)} z^{(l)},\ldots] \not\sim
N_2^{(\sigma)}\,[\ldots,\nabla^{(l)} z^{(l)},\ldots]$, and from
(\ref{eq:6}) that $M_1^{(\sigma-1)}\,[\ldots,\nabla^{(l)}
z^{(l)},\ldots] \ne M_2^{(\sigma-1)}\,[\ldots,\nabla^{(l)}
z^{(l)},\ldots]$.

Note that the passage from $M_i^{(\sigma)}$ to $M_i^{(\sigma-1)}$ can
be repeated to pass from $M^{(\sigma-1)}$ to $M^{(\sigma-2)}$ while
maintaining the relations $\ne$ and $\sim$ between the components.
This corresponds to a proof by induction on $\sigma$ of the property
that $(M_1^{(l)}\,[\ldots,\nabla^{(f)} z^{(f)},\ldots],
M_2^{(l)}\,[\ldots,\nabla^{(f)} z^{(f)},\ldots])$ continues to be a
chain associated with $(M_1^{(0)}\,[\ldots,\nabla^{(f)} z^{(f)},\ldots],
M_2^{(0)}\,[\ldots,\nabla^{(f)} z^{(f)},\ldots])$.
Since the passage from the pair $(N_1^{(0)}, N_2^{(0)})$ to the previous pair
occurs with operations allowed by Theorem~\ref{thm:1},
Lemma~\ref{lem:2} is proved.

\end{proof}

The following example shows the need for inequality between principal
variables of a chain associated with two given formulas for
Lemma~\ref{lem:2} to be valid.

\begin{ex}
  \label{ex:2}
  Let
  \begin{align*}
    N_1^{(0)} & \equiv M_1^{(0)} [y_1] \equiv y_1 y_2\,(\underline{y_1 \underline{y_2} y_3}) \\
    N_2^{(0)} & \equiv M_2^{(0)} [y_1] \equiv y_1 y_2\,(\underline{y_1 \underline{y_3} y_3})
  \end{align*}
  For the sake of brevity we have underlined $N_i^{(1)} \equiv
  M_i^{(1)}$ and $N_i^{(2)} \equiv M_i^{(2)}$. We have $\sigma = 2$
  and $N_1^{(0)} \ne N_2^{(0)}$, $M_1^{(2)} \not\sim M_2^{(2)}$.

  However we note that
  \begin{align*}
    & M_i^{(1)} [\mathrm{K}\mathrm{K}\,y_1] = N_i^{(2)} && \text{but
      unfortunately}\quad M_i^{(0)}\,[\mathrm{K}\mathrm{K}\,y_1] = y_2 \\
    & M_i^{(0)} [\mathrm{K}^2\mathrm{I}\,y_1] =
      N_i^{(1)} [\mathrm{K}^2\mathrm{I}\,y_1] && \text{but
      unfortunately}\quad N_i^{(1)}\,[\mathrm{K}^2\mathrm{I}\,y_1] = y_3,
  \end{align*}
that is, the presence of two equal principal variables in different
pairs of the chain can generate a certain incompatibility from the
selectors that one would like to use.

\end{ex}

To complete the proof of Theorem~\ref{thm:1} it is sufficient to prove the
following.

\begin{lemma}
  \label{lem:3}
  Let $N_1^{(0)} \ne N_2^{(0)}$ and a chain $(M_1^{(l)},M_2^{(l)})$ ($l
  = 0,\ldots,\sigma$) is associated to $(N_1^{(0)}, N_2^{(0)})$,
  respectively with principle variables
  $z^{(0)},\ldots,z^{(\sigma-1)},z_i^{(\sigma)}$, $i = 1,2$, whose
  components are written as
  \begin{equation*}
    M_i^{(l)}\,[z^{(0)},\ldots, z^{(\sigma-1)},z_1^{(\sigma)}, z_2^{(\sigma)}]\enspace.
  \end{equation*}

  Then there exist 3 sufficiently large integers $h_1,h_2,h_3$, such that, placing
  \begin{equation*}
    \nabla_h \equiv \lambda u\lambda t_1\ldots \lambda
    t_{h+1}\,(t_{h+1}\,u\,t_1\ldots t_h)
  \end{equation*}
  it holds that:
  \begin{enumerate}[(i)]
  \item If $(\underline{M}_1^{(l)},\underline{M}_2^{(l)})$ is the pair
    uniquely associated with
    \begin{equation*}
      (M_1^{(l)}\,[\nabla_{h_3}z^{(0)},\ldots,
    \nabla_{h_3}z^{(\sigma-1)},\nabla_{h_1}z_1^{(\sigma)},
    \nabla_{h_2}z_2^{(\sigma)}],M_2^{(l)}[\ldots]),
  \end{equation*}
  we have that
    $(\underline{M}_1^{(l)},\underline{M}_2^{(l)}), l =
    0,\ldots,\sigma$ is a chain associated with
    $(\underline{M}_1^{(l)}[\ldots],\underline{M}_2^{(l)}[\ldots])$.
  \item The principle variables
    $\underline{z}^{(0)},\ldots,\underline{z}^{(\sigma-1)},\underline{z}^{(\sigma)}$
    of $(\underline{M}_1^{(l)},\underline{M}_2^{(l)})$ are all
    distinct from each other and are distinct from each $z^{(l)}$, $z_i^{(\sigma)}$.
  \end{enumerate}
\end{lemma}

\begin{ex}[Continuation of Example~\ref{ex:2}]
  It is sufficient to choose $h_1 = 0$, $h_2 = 1$, $h_3 = 2$. By
  performing the substitutions we obtain
  \begin{align*}
    & M_1^{(0)} [\nabla_2y_1,\nabla_0y_2,\nabla_1y_3] = \lambda
      t_3\,(t_3\,y_1(\nabla_0y_2)\,\lambda
      t_3\,(t_3\,y_1(\nabla_0y_2)\,(\nabla_1y_3))) \\
    & M_2^{(0)} [\nabla_2y_1,\nabla_0y_2,\nabla_1y_3] = \lambda
      t_3\,(t_3\,y_1(\nabla_0y_2)\,\lambda
      t_3\,(t_3\,y_1(\nabla_1y_3)\,(\nabla_1y_3))) \\
    & \underline{M}_1^{(0)} = x_1\,y_1\,(\nabla_0y_2)\,\lambda
      t_3\,(t_3\,y_1\,(\nabla_0 y_2)\,(\nabla_1 y_3)) \\
    & \underline{M}_2^{(0)} = x_1\,y_1\,(\nabla_0y_2)\,\lambda
      t_3\,(t_3\,y_1\,(\nabla_1 y_3)\,(\nabla_1 y_3)) \\
    & \underline{M}_1^{(1)} = x_2\,y_1\,\lambda t_1\,(t_1\,
      y_2)\,(\nabla_1 y_3) \\
    & \underline{M}_2^{(1)} = x_2\,y_1\,\lambda t_1\,\lambda t_2\,(t_2\,
      y_3\,t_1)\,(\nabla_1 y_3) \\
    & \underline{M}_1^{(2)} = x_3\,y_2 \\
    & \underline{M}_2^{(2)} = x_4\,y_3\,x_3
  \end{align*}
\end{ex}
Having now all the principle variables of the chain distinct from each
other, we can apply Lemma~\ref{lem:2} to obtain
\begin{equation*}
  N_i^{(0)} [\nabla_2y_1,\nabla_0y_2,\nabla_1y_3]\,(\mathrm{K}^3\mathrm{I}\,
  x_1)\,(\mathrm{K}^2\mathrm{K} x_2)\,(\mathrm{K}^2\mathrm{K} x_3)\,
  (\mathrm{K}^4\mathrm{I}\,x_4)\,v_1v_2 = v_i
\end{equation*}
that is
\begin{equation*}
  N_i^{(0)} [\nabla_2y_1,\nabla_0y_2,\nabla_1y_3]\,
  (\mathrm{K}^2\mathrm{I})\,(\mathrm{K}\mathrm{K})\,
  (\mathrm{K}\mathrm{K})\,(\mathrm{K}^3\mathrm{I})\,v_1v_2 = v_i\enspace.
\end{equation*}

\begin{proof}[Proof (of Lemma~\ref{lem:3})]
  If any of $M_i^{(l)}\,[z^{(0)},\ldots,z_2^{(\sigma)}]$ has the
  structure of (\ref{eq:2}), then for sufficiently large $h_1,h_2,h_3$,
  $M_i^{(l)}\,[\nabla_{h_3}z^{(0)},\ldots,\nabla_{h_2}z_2^{(\sigma)}]$
  has the following structure:
  \begin{equation}
    \label{eq:8}
    \lambda t_{g_i+1}\ldots \lambda
    t_{h_j+1}\,(t_{h_j+1}\,z\,Z_1^{(i)}\cdots Z_{g_i}^{(i)}\,
    t_{g_i+1}\ldots t_{h_j}),\qquad j = 1,2,3;\quad i = 1,2
  \end{equation}
  where $Z_k^{(i)}$ represent the result of substitutions
  $[z^{(0)} \mapsto \nabla_{h_3}z^{(0)}],\ldots,
  [z_2^{(\sigma)} \mapsto \nabla_{h_2}z_2^{(\sigma)}]$ in $Y_k^{(i)}$.

  Since $M_1^{(l)} \sim M_2^{(l)}$ for $l < \sigma$, thus $g_1=g_2=g$
  in these formulas; meanwhile being $M_1^{(\sigma)} \not\sim
  M_2^{(\sigma)}$ it could be $g_1\ne g_2$.

  Therefore the grade of each
  $M_i^{(l)}\,[z^{(0)},\ldots,z^{(\sigma)}]$ can be indicated by
  $g_0,g_1,\ldots g_{\sigma-1}$, $g_1^{(\sigma)},g_2^{(\sigma)}$. Let
  $\overline{g}$ be the maximal of these grades.

  The choice $h_1 = h_2 = h_3 = \overline{g}$ guarantees that
  $\underline{z}^{(0)},\ldots, \underline{z}^{(\sigma-1)}$ in
  $\underline{M}_i^{(l)}$ are distinct from each other and are
  distinct from each
  $z^{(0)},\ldots z_1^{(\sigma)},z_2^{(\sigma)}$; but it does not
  guarantee $\underline{z}_1^{(\sigma)} \neq
  \underline{z}_2^{(\sigma)}$ in the case where
  $g_1^{(\sigma)} = g_2^{(\sigma)} = g^{(\sigma)}$ (the case of Example~\ref{ex:2}).

  A simple inspection of (\ref{eq:8}) permits to assert
  \begin{equation*}
    h_1 - g^{(\sigma)} \ne h_2 - g^{(\sigma)} \rightarrow
    \underline{z}_1^{(\sigma)} \ne \underline{z}_2^{(\sigma)}
  \end{equation*}
  since in this case it is sufficient to choose
  \begin{equation*}
    h_1 = h_2 = \overline{g}\quad,\quad h_2 = \overline{g} + 1\enspace.
  \end{equation*}

  Note that the solution of Example~\ref{ex:2} is not valid in general
  because it is possible that $z_1^{(\sigma)}$ or $z_2^{(\sigma)}$ is
  equal to some $z^{(l)}$ and $g_l = \overline{g}$.

  We have thus shown that $\underline{z}_1^{(\sigma)} \ne
  \underline{z}_2^{(\sigma)}$ and that those $\underline{z}$ are all
  distinct from each other and are distinct from
  $z^{(0)},\ldots,z_1^{(\sigma)},z_2^{(\sigma)}$.
  In particular, it is also that $\underline{M}_1^{(\sigma)} \not\sim
  \underline{M}_2^{(\sigma)}$.
  Examining~(\ref{eq:8}) one can see that $\underline{M}_1^{(l)} \sim
  \underline{M}_2^{(l)}$ for $l = 0,\ldots,\sigma-1$.

  A reasoning similar to that used to demonstrate Lemma~\ref{lem:2}
  allows us to deduce $\underline{M}_1^{(\sigma-1)} \ne
  \underline{M}_2^{(\sigma-1)}$ from $\underline{M}_1^{(\sigma)} \ne
  \underline{M}_2^{(\sigma)}$ and therefore by induction on $\sigma$
  that $(\underline{M}_1^{(l)},\underline{M}_2^{(l)}$ is a chain
  associated with $(M_1^{(0)}\,[\nabla_{h_3}
  z^{(0)},\ldots,\nabla_{h_2} z_2^{(\sigma)}], M_2^{(0)}[\ldots])$.
  Since the passage from $(N_1^{(0)},N_2^{(0)})$ to the pair just
  written occurs with operations allowed by Theorem~\ref{thm:1},
  Lemma~\ref{lem:3} is proved since it is reduced to Lemma~\ref{lem:2},
  and consequently Theorem 1 is also proved.
\end{proof}

The proof method for Theorem~\ref{thm:1} is constructive, therefore it
suggests an algorithm for constructing the formulas.

The suggested algorithm, however, is not the most efficient one.

We now give some additional information that we did not want to
include in the demonstration so as not to complicate it further.

While it does not seem possible to improve the results of
Lemma~\ref{lem:1} and~\ref{lem:2}, as far as Lemma~\ref{lem:3} is
concerned it is clear that a transformation $\nabla_h$ is
indispensable if there exists a class of equal variables $z^{(l)}$,
and only in the case that, as in Example~\ref{ex:2}, it happens that
there exist two or more incompatible selectors.
For example no transformation $\nabla_h$ is necessary for
\begin{align*}
  N_1 &= \lambda x\,(x\,(x\,(x\,y_1))) \\
  N_2 &= \lambda x\,(x\,(x\,y_2))
\end{align*}
in which
$N_i\,[\mathrm{K}\mathrm{K}y_1,\mathrm{K}^2\mathrm{I}\,y_2]\,\mathrm{I}\,v_1v_2
= v_i$, while in the case
\begin{align*}
  N_1 &= \lambda x\lambda y\,(x\,(x\,(x\,y))) = \underline{3} \\
  N_2 &= \lambda x\lambda y\,(x\,(x\,y)) = \underline{2}
\end{align*}
the transformation $\nabla_0$ is sufficient (while the theory
indicates $\nabla_1$):
\begin{equation*}
  N_i(\nabla_0 \mathrm{I})\,(\mathrm{K}^3\mathrm{I})\,(\mathrm{K} v_1) v_2 = v_i
\end{equation*}

It can be seen from these examples that Lemma~\ref{lem:3} is
susceptible to some improvement, with respect to the efficiency of the
associated algorithm.

\section{Further developments of Theorem~\ref{thm:1} and Corollary~\ref{cor:1}}

The method of proof of theorem 1 allows us to easily prove the following

\begin{corollary}
  \label{cor:2}
  For each combinator $F\in\mathcal{N}$, if $y$ occurs in $F$ then
  there exist 2 combinators $\Delta$ and $\nabla$ such that
  \begin{equation*}
    \Delta \circ F \circ \nabla = \mathrm{I}
  \end{equation*}
  and that
  \begin{equation*}
    \Delta \equiv \langle G_1,\ldots, G_s \rangle,\qquad G_j \equiv
    \mathrm{K}^{a_j}\mathrm{K}^{b_j},\qquad \Delta\equiv \mathrm{K}^g
  \end{equation*}
  where $a_j, b_j, g, s \geqslant 0$ are integers.
\end{corollary}

\begin{proof}
  Let
  \begin{align*}
    N_1\,[y_1,y_2] & \equiv F\,y_1 \\
    N_2\,[y_1,y_2] & \equiv F\,y_2
  \end{align*}
It is obvious that $N_1 \ne N_2, N_1 \sim N_2$ and there exists a
chain $(M_1^{(l)},M_2^{(l)})$, $l = 0,\ldots,\sigma$ with
$M_1^{(\sigma)} \not\sim M_2^{(\sigma)}$.
It is almost obvious that $M_i^{(\sigma)}$ has the form
\begin{equation*}
  M_i^{(\sigma)}\,[y_1,y_2] = y_i\,Y_1\cdots Y_g
\end{equation*}
and that the $y_i$ do not coincide with any principle variables of
$M_i^{(0)},\ldots M_i^{(\sigma-1)}.$

Therefore we have
\begin{equation*}
  M_i^{(\sigma)}\,[\mathrm{K}^g y_1,\mathrm{K}^g y_2] = y_i
\end{equation*}
from which, by applying Theorem~\ref{thm:1} and Corollary~\ref{cor:1}
regarding the passage from $N_i$ to $M_i^{(\sigma)}$, we can write
\begin{equation*}
  N_i\,[\mathrm{K}^g y_1,\mathrm{K}^g y_2]\,G_1\cdots G_s = y_i
\end{equation*}
or return to the original writing
\begin{equation}
  \label{eq:9}
  F(\mathrm{K}^g y)\,G_1\cdots G_s = y
\end{equation}

Now from Corollary~\ref{cor:1} we know that $G_j$ are combinators,
from the proof of Theorem~\ref{thm:1} it turns out that $G_j =
\mathrm{I} = \mathrm{K}^0 \mathrm{K}^0$ or it is a selector of the
form $\mathrm{K}^a\mathrm{K}^b$. Equation~(\ref{eq:9}) can be
rewritten to
\begin{equation*}
  \langle G_1,\ldots G_s\rangle\,(F(\mathrm{K}^g y)) = y
\end{equation*}
from which, abstracting with respect to $y$ and introducing the
composition operator $\circ$, we obtain
\begin{equation*}
  \langle G_1,\ldots G_s\rangle \circ F \circ \mathrm{K}^g = \mathrm{I}
\end{equation*}
which proves the corollary.
\end{proof}

Theorem~\ref{thm:1} and its corollaries can give some information
about the algebraic structure of combinators regardless of whether or
not they possess normal form.

In fact, the proof of Theorem~\ref{thm:1} essentially depends on the
existence of the chain $(M_1^{(l)},M_2^{(l)})$ for normal formulas
distinct from each other.
For the existence of the chain it is sufficient but not necessary that
the starting formulas are normal.
To see this, it is enough to think about the fact that each formula
canceled by the process might not have a normal form without affecting
the final result.
We now show another result on combinators not necessarily having
normal form.

\begin{corollary}
  \label{cor:3}
  Let $X,c_1,c_2$ be combinators.

  If $X c_1 \ne X c_2$ and these formulas both admit normal form,
  there exists at least a third combinator $c$ such that
  \begin{equation*}
    X c \ne X c_1\quad,\quad X c \ne X c_2
  \end{equation*}
\end{corollary}

\begin{proof}
  From Corollary~\ref{cor:1} it follows that there exists $\Delta$
  such that
  \begin{align*}
    \Delta\,(X c_1) &= c_2 \\
    \Delta\,(X c_2) &= c_1
  \end{align*}

  Let $b = \theta\,(X\circ \Delta)$ where for $\theta$ the property
  $\theta x = x\,(\theta x)$ holds.

  We have therefore
  \begin{equation*}
    b = X\,(\Delta\,(\theta\,(X\circ\Delta))) = X\,(\Delta b)\enspace.
  \end{equation*}

  Now if it were $b = X c_1$ one would have
  \begin{equation*}
    X c_1 = X\,(\Delta\,(X c_1)) = X c_2
  \end{equation*}
  in contradiction to the hypothesis; And if it were $b = X c_2$ one
  would have
  \begin{equation*}
    X c_2 = X\,(\Delta\,(X c_2)) = X c_1
  \end{equation*}
  in contradiction; therefore $b \ne X c_1, b \ne X c_2$.
  By choosing $c = \Delta b$ now the thesis follows.
\end{proof}

\bibliographystyle{amsalpha}
\bibliography{Boehm}

\end{document}